\documentclass{article}

\usepackage[english]{babel}
\usepackage{amsthm}
\usepackage[nottoc]{tocbibind}
\newtheorem{defin}{Definition}
\newtheorem{lem}{Lemma}
\newtheorem{thm}{Theorem}
\newtheorem{cor}{Corollary}

\theoremstyle{definition}
\newcommand{\Nat}{{\mathbb N}}

\usepackage{amsmath, amssymb}
\newcommand{\psn}[1]{\mathrm{psn}{(#1)}}

\newcommand{\vei}{\varepsilon_{\omega}}
\newcommand{\ve}{\varepsilon}
\newcommand{\LD}{\mathrm{LD}}

\newcommand{\ACA}{\mathrm{ACA}}
\newcommand{\GO}{\mathrm{GO}}
\newcommand{\RT}{\mathrm{RT}}

\newcommand{\KS}{\mathrm{KS}}

\newcommand{\tow}{\mathrm{tow}}

\title{Partitioning $\alpha$-large sets for $\alpha<\vei$}
\author{Michiel De Smet\footnote{Aspirant Fonds Wetenschappelijk Onderzoek (FWO) - Flanders} \  and Andreas Weiermann\\
Ghent University \\
Building S22 \\
Krijgslaan 281\\
B 9000 Gent\\
mmdesmet@cage.ugent.be\\ 
weierman@cage.ugent.be} 
\date{\today}

\begin{document}
\maketitle
\begin{abstract}
We generalise the results by Bigorajska and Kotlarski about partitioning $\alpha$-large sets, by extending the domain up to ordinals below $\ve_{\omega}$. These results will be very useful to give a miniaturisation of the infinite Ramsey Theorem.
 \end{abstract}

\section{Introduction}
McAloon, in 1985, writes down in \cite{McAloon1985}:
\begin{center}
\emph{It would be interesting to develop proofs of these results with the\\  ``direct'' method of $\alpha$-large sets of Ketonen-Solovay \cite{Ketonen1981}.}
\end{center}
The results he speaks about concern Paris-Harrington incompleteness using finite Ramsey theorems. In that paper he gives a first order axiomatization of the first order consequences of $\ACA_{0}+\RT$, where $\RT$ stands for the infinite version of Ramsey's theorem.
\par
Ketonen and Solovay used the $\alpha$-largeness in their paper on rapidly growing Ramsey functions (\cite{Ketonen1981}), in which they extend the famous result of Paris and Harrington.  They established sharp upper and lower bounds on the Ramsey function by purely combinatorial means.
Bigorajska and Kotlarski generalised their ideas to ordinals below $\ve_{0}$ and obtained several results on partitioning $\alpha$-large sets (see \cite{Bigorajska1999, Bigorajska2002, Bigorajska2006}). This paper is dedicated to generalise those latter results in order to allow ordinals  up to $\ve_{\omega}$. 
\par
After giving definitions and preliminary results, we focus on the estimation lemma. This lemma will be crucial to partition $\alpha$-large sets in the final section. In a following paper we show how these results can be used to prove a miniaturisation of the infinite Ramsey Theorem.

\section{Definitions and preliminary results}
As mentioned in the introduction, our results are generalisations of the ones of Bigorajska and Kotlarski. We assume that the reader has a copy of their papers (\cite{Bigorajska1999,Bigorajska2006}) in hand. Moreover, to avoid repetition as much as possible, we define only those notions which are not yet defined in \cite{Bigorajska1999,Bigorajska2006}. Towards the end of this section proofs often are very similar to the ones of Bigorajska and Kotlarski, so we refer to their paper if possible.

Henceforth $\ve_{-1}$ will sometimes be used to denote $\omega$, for the sake of generality of notation. We will only consider ordinals below $\vei$ and use Greek letters to write them down. Letters of the Latin alphabet will denote natural numbers. Unless clearly stated otherwise, these agreements hold for the rest of the paper. 
\par
We first need to generalise definitions given in  \cite{Bigorajska2006}. We introduce a standard notation for each ordinal and define the pseudonorm of an ordinal $\alpha$ as the maximum of the greatest natural number occurring in this notation and its height $h(\alpha)$. 
\begin{defin} Let $\alpha<\vei$.
\begin{enumerate}
\item $l(\alpha):=\min \{n \in \Nat : \alpha < \ve_{n} \}$;
\item $h(\alpha):= \min \{n \in \Nat : \alpha < \tow_{n}( \ve_{l(\alpha)-1} )\}$;
\item We say $\alpha$ is written in normal form to the base $\ve_{m}$ if 
$$\alpha = \ve_{m}^{\alpha_{0}}\cdot\xi_{0}+\ve_{m}^{\alpha_{1}}\cdot\xi_{1}+\ldots+\ve_{m}^{\alpha_{s}}\cdot\xi_{s},$$
for some $\alpha_{1}>\alpha_{1}>\ldots>\alpha_{s}$ and $0<\xi_{i}<\ve_{m}$, for $0\leq i \leq s$. If $m=l(\alpha)-1$, then we say $\alpha$ is written in normal form and write down $\alpha =_{NF} \ve_{m}^{\alpha_{0}}\cdot\xi_{0}+\ve_{m}^{\alpha_{1}}\cdot\xi_{1}+\ldots+\ve_{m}^{\alpha_{s}}\cdot\xi_{s}$.
\item If $\alpha$ is written in normal form, define
$$\psn{\alpha}:=
\begin{cases}
\max \{ 
h(\alpha),
\psn{\alpha_{0}},\ldots,\psn{\alpha_{s}},\psn{\xi_{0}},\ldots,\psn{\xi_{s}}\} \\
\hspace{7.2cm}\text{ if } \alpha\geq \omega\\
\alpha &\\
\hspace{7.2cm} \text{ if } \alpha< \omega\\
\end{cases}.$$
\end{enumerate} 
\end{defin}
If $\alpha$ is (written) in normal form to the base $\ve_{m}$, i.e. $\alpha = \ve_{m}^{\alpha_{0}}\cdot\xi_{0}+\ldots+\ve_{m}^{\alpha_{s}}\cdot\xi_{s}$, we sometimes use the notion \emph{short normal form} to speak about $\alpha = \ve_{m}^{\alpha_{0}}\cdot\xi_{0}+\psi,$ or about  $\alpha = \psi+\ve_{m}^{\alpha_{s}}\cdot\xi_{s},$ depending on the situation.
\par
The notions $\beta \gg \alpha$ and $\beta \ggg \alpha$, are defined as in the literature, but we use $\beta$ in normal form and $\alpha$ in normal form to the base $\ve_{\l(\beta)-1}$, instead of the Cantor normal form for both ordinals.
\begin{lem} \label{lem-bases}
Let $\alpha<\vei$ and $m\geq l(\alpha)-1$ Let  $\alpha=\ve_{m}^{\alpha_{0}}\cdot\xi_{0}+\ldots+\ve_{m}^{\alpha_{s}}\cdot\xi_{s}$ be written in normal form to the base $\ve_{m}$.  
Then 
$$\psn{\alpha}\geq\max\{\psn{\alpha_{0}},\ldots, \psn{\alpha_{s}}, \psn{\xi_{0}},\ldots,\psn{\xi_{s}}\}.$$
\end{lem}

\begin{proof}
We consider two cases: $m=l(\alpha)-1$ and $m>l(\alpha)-1$. 
\par 
\textsc{Case 1}.  $m=l(\alpha)-1$. Then $\alpha=_{NF}\ve_{m}^{\alpha_{0}}\cdot\xi_{0}+\ldots+\ve_{m}^{\alpha_{s}}\cdot\xi_{s}$ and the statement follows by the definition of $\psn{\alpha}$.
\par
\textsc{Case 2}.  $m>l(\alpha)-1$. Then  $\alpha=\ve_{m}^{0}\cdot\alpha$ and the statement becomes trivial. 
\end{proof}

For each limit $\lambda$ we define a sequence $\lambda[n]$ of ordinals converging to $\lambda$ from below. In the next definition we only cover those cases not dealt with in \cite{Bigorajska2006}.
\begin{defin}
Let $\lambda$ be a limit ordinal below $\vei$, written in normal form, and $n$ be a natural number.
\begin{enumerate}
\item If $\lambda=\ve_{m},$ with $m\geq0$, then $\lambda[n]=\tow_{n}(\ve_{m-1})$;
\item If $\lambda=\ve_{m}^{\alpha+1},$ with $m\geq-1$, then $\lambda[n]=\ve_{m}^{\alpha}\cdot \ve_{m}[n]$;
\item If $\lambda=\ve_{m}^{\psi},$ with $m\geq-1$ and $\psi$ a limit, then $\lambda[n]=\ve_{m}^{\psi[n]}$;
\item If $\lambda=\ve_{m}^{\psi}\cdot (\alpha+1),$ with $m\geq0$, then $\lambda[n]=\ve_{m}^{\psi}\cdot \alpha+\ve_{m}^{\psi}[n]$;
\item If $\lambda=\ve_{m}^{\psi}\cdot \xi,$ with $m\geq0$ and $\xi$ limit, then $\lambda[n]=\ve_{m}^{\psi}\cdot \xi[n]$;
\item If $\lambda =\ve_{m}^{\lambda_{0}}\cdot\xi_{0}+\ldots+\ve_{m}^{\lambda_{s}}\cdot\xi_{s}$, with $m\geq0$, then $\lambda[n]=\ve_{m}^{\lambda_{0}}\cdot\xi_{0}+\ldots+(\ve_{m}^{\lambda_{s}}\cdot\xi_{s})[n]$.
\end{enumerate}
We shall call the sequence $\lambda[n]$ the fundamental sequence for $\lambda$.
\end{defin}

As we have extended our notion of fundamental sequence, now we can state the following lemma.

\begin{lem} \label{lem-fundpsnfull}For every limit $\beta<\vei$ we have:
$$(\forall \alpha < \beta)( \forall n>1)(\psn{\alpha}<n \Rightarrow \alpha < \beta[n]).$$
\end{lem}
\begin{proof}  By induction on  $\beta$. The case $\beta=0$ is trivial. Assume the statement is proved for $\alpha<\beta$. Let $\beta=\varepsilon_{m}^{\beta_{0}}\cdot\beta_{1}+\beta_{2}$  (in short normal form) and $\alpha=\varepsilon_{m}^{\alpha_{0}}\cdot\alpha_{1}+\alpha_{2}$ (in short normal form to the base $\ve_{m}$). The proof reduces to the following case study.
\begin{enumerate}
\item $\beta_{2}\neq 0:$
\begin{enumerate}
\item $\beta_{0}>\alpha_{0}:$ $\alpha< \ve_m^{\beta_{0}}\leq \ve_m^{\beta_{0}}\cdot \beta_{1} + \beta_{2}[n]=\beta[n]$.
\item $\beta_{0}=\alpha_{0}:$
\begin{enumerate}
\item $\beta_{1}>\alpha_{1}:$ $\alpha< \ve_m^{\beta_{0}}\cdot \beta_{1} \leq \ve_m^{\beta_{0}}\cdot \beta_{1} + \beta_{2}[n]=\beta[n]$.
\item $\beta_{1}=\alpha_{1}:$ We must have $\alpha_{2}<\beta_{2}$. Since $\beta_{2}$ is a limit ordinal, we can apply the induction hypothesis to obtain $\alpha_{2}<\beta_{2}[n]$ and so 
$$\alpha = \ve_m^{\alpha_{0}}\cdot\alpha_{1}+\alpha_{2}= \ve_m^{\beta_{0}}\cdot\beta_{1}+\alpha_{2} <  \ve_m^{\beta_{0}}\cdot\beta_{1}+\beta_{2}[n]=\beta[n].$$
\end{enumerate}
\end{enumerate}
\item $\beta_{2}=0:$
\begin{enumerate}
\item $\beta_{0}>\alpha_{0}:$
\begin{enumerate}
\item $\beta_{1}>1:$ Then $\beta_{1}[n]\geq1$ and we have $\alpha<\ve_m^{\beta_{0}}\leq \ve_m^{\beta_{0}}\cdot \beta_{1}[n]=\beta[n]$.
\item $\beta_{1}=1:$ 
\begin{enumerate}
\item $\beta_{0}=\beta_{0}'+1:$ Then $\alpha_{0}\leq \beta_{0}'$. Since $\alpha_{1}<\ve_m$ and $\psn{\alpha_{1}}\leq \psn{\alpha} < n$, we have $\alpha_{1}<\tow_{n}(\ve_{m-1})$. All together this yields $\alpha=\ve_m^{\alpha_{0}}\cdot\alpha_{1}+\alpha_{2}<\ve_m^{\beta_{0}'}\cdot \tow_{n}(\ve_{m-1})=\beta[n]$.
\item $\beta_{0}$ limit : Since $\alpha_{0}<\beta_{0}$ and $\psn{\alpha_{0}}\leq \psn{\alpha} < n$, the induction hypothesis yields $\alpha_{0}<\beta_{0}[n]$, and thus $\alpha=\ve_m^{\alpha_{0}}\cdot\alpha_{1}+\alpha_{2}<\ve_m^{\beta_{0}[n]}=\beta[n]$.

\end{enumerate}
\end{enumerate}
\item $\beta_{0}=\alpha_{0}:$ We must have $\beta_{1}>\alpha_{1}$.
\begin{enumerate}
\item $\beta_{1}=\beta_{1}'+1:$ Then $\beta_{1}'\geq \alpha_{1}$.
\begin{enumerate}
\item $\beta_{0}=\beta_{0}'+1: $ Since $\alpha_{2}< \ve_m^{\beta_{0}}$ and $\psn{\alpha_{2}}\leq \psn{\alpha} < n$, we have $\alpha_{2}<\ve_m^{\beta_{0}'}\cdot \tow_{n}(\ve_{m-1})$. So, 
\begin{align*}
\alpha&=\ve_m^{\alpha_{0}}\cdot\alpha_{1}+\alpha_{2}\\
&=\ve_m^{\beta_{0}}\cdot\alpha_{1}+\alpha_{2}\\
& < \ve_m^{\beta_{0}}\cdot\beta_{1}'+\ve_m^{\beta_{0}'}\cdot\tow_{n}(\ve_{m-1})\\
&=\beta[n].
\end{align*}

\item $\beta_{0}$ limit : Since $\alpha_{2}<\ve_m^{\alpha_{0}}=\ve_m^{\beta_{0}}$, we can write $\alpha_{2}=\ve_m^{\xi_{0}}\cdot\xi_{1}+\xi_{2}$, with $\xi_{0}<\beta_{0}$. If we combine this with $\psn{\xi_{0}} \leq \psn{\alpha_{2}}< n$, then the induction hypothesis implies $\xi_{0}<\beta_{0}[n]$, and so $\alpha_{2}<\ve_m^{\beta_{0}[n]}$. Finally we obtain
$$\alpha=\ve_m^{\alpha_{0}}\cdot\alpha_{1}+\alpha_{2}=\ve_m^{\beta_{0}}\cdot\alpha_{1}+\alpha_{2}< \ve_m^{\beta_{0}}\cdot\beta_{1}'+\ve_m^{\beta_{0}[n]}=\beta[n].$$
\end{enumerate}
\item $\beta_{1}$ limit : Similarly as above the induction hypothesis implies $\beta_{1}[n]>\alpha_{1}$ and so $\alpha=\ve_m^{\alpha_{0}}\cdot\alpha_{1}+\alpha_{2}<\ve_m^{\beta_{0}}\cdot\beta_{1}[n]=\beta[n]$.
\end{enumerate}
\end{enumerate}
\end{enumerate}
\end{proof}

\begin{defin}
For $\alpha, \beta < \vei$ we write $\beta \Rightarrow_{n} \alpha$ if there exists a finite sequence $\alpha_{0},\ldots , \alpha_{k}$ of ordinals such that $\alpha_{0} = \beta , \alpha_{k} = \alpha$ and $\alpha_{m+1}=\alpha_{m}[n]$ for every $m < k$.
\end{defin}

\begin{lem} \label{lem-fundpsnseq}For every $\beta<\vei$ we have:
$$(\forall \alpha < \beta)( \forall n>1)((\psn{\alpha}<n) \Rightarrow (\beta\Rightarrow_{n}\alpha)).$$
\end{lem}
\begin{proof} The proof goes by transfinite induction on $\beta$.
\par
If $\beta$ equals zero, then the statement is trivial.
\par
Let $\beta=\gamma+1$. If $\alpha=\gamma$, then $\alpha=\beta[n]$ and so $\beta\Rightarrow_{n}\alpha$. If $\alpha < \gamma$, then apply the induction hypothesis to obtain $\gamma \Rightarrow_{n} \alpha$. Now $\beta[n]=\gamma$ yields $\beta\Rightarrow_{n}\alpha$.
\par
Let $\beta$ be a limit ordinal. Due to the previous lemma we know $\alpha < \beta[n]$. Apply the induction hypothesis to obtain $\beta[n] \Rightarrow_{n} \alpha$, and so $\beta\Rightarrow_{n}\alpha$.
\end{proof}

\begin{lem}\label{lem-Nmk} Let $a, k, m \in \Nat$, such that $1 \leq k\leq a$. Then
$$|\{ \alpha < \tow_{k}(\ve_{m}) : \psn{\alpha}\leq a\}|=\tow_{k}\underbrace{(\tow_{a}(\ldots(\tow_{a}}_{m+1\text{ times}} (a+1))\ldots)).$$
\end{lem}
\begin{proof}
Take any $a\in \Nat$, with $1 \leq a$. The proof proceeds by double induction on $m$ and $k\leq a$. Let $N^{m}_{k}(a)=\{ \alpha < \tow_{k}(\ve_{m}) : \psn{\alpha}\leq a\}$. We want to prove that 
$$
|N^{m}_{k}(a)|=\tow_{k}\underbrace{(\tow_{a}(\ldots(\tow_{a}}_{m+1\text{ times}} (a+1))\ldots)).
$$

\begin{enumerate}
\item $m=0$
\begin{enumerate}
\item $k=1:$ We have
\begin{align*}
|N^{0}_{1}(a)|&=|\{ \alpha < \ve_{0} : \psn{\alpha}\leq a\}|\\
&=|\{ \alpha < \tow_{a}(\omega) : \psn{\alpha}\leq a\}|\\
&=\tow_{a}(a+1)\\
&=\tow_{1}(\tow_{a}(a+1)),
\end{align*}
where the last equality is a result from the proof of Lemma 3.5 in \cite{Bigorajska2006}.
\item $k\to k+1:$ Assume the statement for $k$ and let $k+1\leq a$ and $N^{0}_{k}(a)=\{\alpha_{0}, \ldots, \alpha_{r}\}$ in decreasing order. Every $\alpha \in N^{0}_{k+1}$ written in normal form to the base $\ve_{0}$ (which is slightly changed by allowing some coefficients to be zero) may be identified with the sequence of its coefficients, whose exponents are $\alpha_{0},\ldots, \alpha_{r}$. It follows that 
\begin{align*}
|N^{0}_{k+1}(a)|&=(|N^{0}_{1}(a)|)^{|N^{0}_{k}(a)|}\\
&=(\tow_{a}(a+1))^{\tow_{k}(\tow_{a}(a+1))}\\
&=\tow_{k+1}(\tow_{a}(a+1))
\end{align*}
\end{enumerate}
\item $m \to m+1:$ Assume the statement is proven for $m$.
\begin{enumerate}
\item $k=1:$ Using the induction hypothesis we obtain
\begin{align*}
|N^{m+1}_{1}(a)|&=|\{ \alpha < \ve_{m+1} : \psn{\alpha}\leq a\}|\\
&=|\{ \alpha < \tow_{a}(\ve_{m}) : \psn{\alpha}\leq a\}|\\
&=\tow_{a}\underbrace{(\tow_{a}(\ldots(\tow_{a}}_{m+1\text{ times}} (a+1))\ldots))\\
&=\underbrace{\tow_{a}(\ldots(\tow_{a}}_{m+2\text{ times}} (a+1))\ldots)\\
&=\tow_{1}(\underbrace{\tow_{a}(\ldots(\tow_{a}}_{m+2\text{ times}} (a+1))\ldots)).
\end{align*}
\item $k\to k+1:$ Reasoning in a similar way as done in \text{1. (b)} and using the induction hypothesis and \text{2. (a)}, we obtain
\begin{align*}
|N^{m+1}_{k+1}(a)|&=(|N^{m+1}_{1}(a)|)^{|N^{m+1}_{k}(a)|}\\
&=(\underbrace{\tow_{a}(\ldots(\tow_{a}}_{m+2\text{ times}} (a+1))\ldots))^{\tow_{k}(\underbrace{\tow_{a}(\ldots(\tow_{a}}_{m+2\text{ times}} (a+1))\ldots))}\\
&=\tow_{k+1}(\underbrace{\tow_{a}(\ldots(\tow_{a}}_{m+2\text{ times}} (a+1))\ldots)),
\end{align*}
which concludes the induction and therefore the proof.
\end{enumerate}
\end{enumerate}
\end{proof}
Remark that, for $k>a$, 
\begin{align*}
N^{m}_{k}(a)&=\{ \alpha < \tow_{k}(\ve_{m}) : \psn{\alpha}\leq a\}\\
&=\{ \alpha < \tow_{a}(\ve_{m}) : \psn{\alpha}\leq a\}\\
&=N^{m}_{a}(a).
\end{align*}
\par
Let $h: \Nat \to \Nat$ be an increasing function. We define the \emph{Hardy hierarchy based on $h$} as follows:
\begin{align*}
h_{0}(x)&=x\\
h_{\alpha+1}(x)&=h_{\alpha}(h(x))\\
h_{\lambda}(x)&=h_{\lambda[x]}(x),
\end{align*}
with $\lambda$ a limit ordinal. In general $h$ will be the successor function $h(x)=x+1$. However, if we consider a specific set $A \subseteq \Nat$, we work with the successor function in the sense of $A$. Thus $h$ denotes the function defined on $A$ which associates with every $a \in A$ (or $a\in A\setminus\{ \max A\}$ if $A$ is finite) the next element of $A$. In this case we  sometimes explicitly write $h^{A}_{\alpha}$.
\begin{lem} \label{lem-h>tower}
For all $m\geq 1$ and $a\geq 1$, we have for all $x>0$,
$$h_{\omega^{2}\cdot2am}(x)\geq \underbrace{\tow_{a}(\ldots(\tow_{a}}_{m \text{ times}} (x+1))\ldots).$$
\end{lem}
\begin{proof} Fix $a \geq 1$.
By induction on $m$.
\par
If $m=1$, then $h_{\omega^{2}\cdot2a}(x)\geq \tow_{a} (x+1),$ by Lemma 4.5 in \cite{Bigorajska2006}.
\par
Assume the statement is proven for $m$, we prove it for $m+1$. The case $m=1$ and the induction hypothesis imply
\begin{align*}
h_{\omega^{2}\cdot2a(m+1)}(x)&\geq h_{\omega^{2}\cdot2am}(h_{\omega^{2}\cdot2a}(x))\\
&\geq h_{\omega^{2}\cdot2am}(\tow_{a}(x+1))\\
&\geq \underbrace{\tow_{a}(\ldots(\tow_{a}}_{m\text{ times}}(\tow_{a} (x+1)+1))\ldots)\\
&\geq \underbrace{\tow_{a}(\ldots(\tow_{a}}_{m+1 \text{ times}} (x+1))\ldots),
\end{align*}
which completes the proof.
\end{proof}

\begin{lem} \label{lem-h}
For every $\alpha < \vei$:
\begin{enumerate}
\item $h_{\alpha}$ is increasing;
\item For every $\beta, b$ if $\alpha \Rightarrow_{b} \beta$ then if $h_{\alpha}(b)$ exists then $h_{\beta}(b)$ exists and $h_{\alpha}(b) \geq h_{\beta}(b)$.
\end{enumerate}
\end{lem}
\begin{proof}
By simultaneous induction on $\alpha$.
\end{proof}

\begin{lem} \label{lem-hepsm}
For all $m\geq 0$ and $a>3$, we have
$$h_{\ve_{m}}(a)>h_{\omega^{2}\cdot2a(m+3)}(a).$$
\end{lem}
\begin{proof}
By induction on $m$.
\par
If $m=0$, then 
\begin{align*}
h_{\ve_{0}}(a) &= h_{\tow_{a}(\omega)}(a) \\
&\geq h_{\omega^{3}+\omega\cdot 3}(a)\\
&= h_{\omega^{2}\cdot8a}(8a)\\
&= h_{\omega^{2}\cdot2a}( h_{\omega^{2}\cdot6a}(8a))\\
&> h_{\omega^{2}\cdot6a}(8a)\\
&\geq h_{\omega^{2}\cdot6a}(a),
\end{align*}
where the last inequality holds because $h_{\alpha}$ is an increasing function. Due to Lemma \ref{lem-fundpsnseq}, we have $\tow_{a}(\omega)\Rightarrow_{a}\omega^{3}+\omega\cdot 3$, since $\psn{\omega^{3}+\omega\cdot 3}=3<a$ and $\omega^{3}+\omega\cdot 3<\tow_{a}(\omega)$. Then the first inequality is implied by Lemma \ref{lem-h}. 
\par
Assume the statement holds for $m$, we prove it for $m+1$. We have
\begin{align*}
h_{\ve_{m+1}}(a) &= h_{\tow_{a}(\ve_{m})}(a) \\
&\geq h_{\ve_{m}\cdot 2}(a) \\
&= h_{\ve_{m}}(h_{\ve_{m}}(a)) \\
&> h_{\ve_{m}}(h_{\omega^{2}\cdot2a(m+3)}(a)) \\
&> h_{\omega^{2}\cdot2a(m+3)}(h_{\omega^{2}\cdot2a(m+3)}(a)) \\
&= h_{\omega^{2}\cdot2a(m+4)}(h_{\omega^{2}\cdot2a(m+2)}(a)) \\
&\geq h_{\omega^{2}\cdot2a(m+4)}(a).
\end{align*}
The first inequality holds because of Lemma \ref{lem-h}. The second inequality is due to the induction hypothesis. The third inequality holds because of the induction hypothesis and Lemma \ref{lem-fundpsnseq}. The last inequality is also due to Lemma \ref{lem-fundpsnseq}.
\end{proof}

\begin{lem} \label{lem-htow}
For all $m\geq 0$ and $a>3$, we have
$$h_{\ve_{m}}(a)>2\underbrace{\tow_{a}(\ldots(\tow_{a}}_{m+2\text{ times}} (a+1))\ldots).$$
\end{lem}
\begin{proof}
Combine Lemma \ref{lem-h>tower} and Lemma \ref{lem-hepsm} with
$$
\underbrace{\tow_{a}(\ldots(\tow_{a}}_{m+3\text{ times}} (a+1))\ldots)>2\underbrace{\tow_{a}(\ldots(\tow_{a}}_{m+2\text{ times}} (a+1))\ldots).
$$
\end{proof}

\section{The estimation lemma}
We need to generalise the notion of natural sum. The idea remains the same as for ordinals below $\ve_{0}$.
\begin{defin}
Let $\alpha$ and $\beta$ be ordinals below $\vei$. We define their natural sum $\alpha \oplus \beta$ by induction on $m=\max \{l(\alpha),l(\beta)\}$:
\par
If $m=0$, then $\alpha, \beta < \ve_{0}$ and the definition is as before. 
\par
Assume we have defined the natural sum for all ordinals $\gamma, \delta$ such that $\max \{l(\gamma),l(\delta)\}=m$ and let $\max \{l(\alpha),l(\beta)\}=m+1$.  Let $\alpha=\ve_{m}^{\alpha_{0}}\cdot \xi_{0}+\ldots+\ve_{m}^{\alpha_{s}}\cdot \xi_{s}$ and $\beta=\ve_{m}^{\beta_{0}}\cdot \eta_{0}+\ldots+\ve_{m}^{\beta_{u}}\cdot \eta_{u}$, both in normal form to the base $\ve_{m}$. Permute items of both of these expansions so that we obtain a nonincreasing sequence of exponents. Then write this sequence as the normal form to the base $\ve_{m}$ of some ordinal which we denote $\alpha \oplus \beta$. While doing so, if $\alpha_{i}=\beta_{j}$ for some $i,j$,  then use the induction hypothesis to obtain $\xi_{i}\oplus\eta_{j}$ as the coefficient of the term with exponent $\alpha_{i}$.
\end{defin}

\begin{lem} If $h_{\beta\oplus\alpha}(a)\downarrow$, then $h_{\beta}\circ h_{\alpha}(a)\downarrow$ and $h_{\beta}\circ h_{\alpha}(a)\leq h_{\beta \oplus \alpha}(a)$. In other words, if a set $A$ is $\beta\oplus\alpha$-large, then there exists $u \in A$, such that $\{a \in A : a \leq u\}$ is $\alpha$-large and $\{ a \in A : u \leq a \}$ is $\beta$-large.
\end{lem}
\begin{proof}
The proof goes in exactly the same way as in \cite{Bigorajska2006}, but one should use the extended notions of normal form and natural sum. 
\end{proof}

Let us identify the greatest ordinal $<\alpha$ whose pseudonorm is $\leq a$. We define the symbol $\GO(a, \alpha)$ for $a>0$ and $\alpha>0$ by induction on $\alpha$.  We let $\GO(a,1)=0$, $\GO(a, \omega)=a$ and $\GO(a,\ve_{m})=\GO(a,\tow_{a}(\ve_{m-1}))$, for $m\geq0$. Remind the notational agreement $\ve_{-1}=\omega$. Other cases are as follows:
$$
\GO(a, \alpha+1)=
\begin{cases}
\alpha & \text{ if }\hspace{0.5 cm}  \psn{\alpha}\leq a,\\
\GO(a,\alpha) & \text{ if }\hspace{0.5 cm}  \psn{\alpha}>a.
\end{cases}
$$
Before giving the limit step we put, for $m\geq -1$ and $\nu>1$, 
$$
\GO(a,\ve_{m}^{\nu})=\ve_{m}^{\GO(a,\nu)}\cdot \GO(a,\ve_{m})+\GO(a,\ve_{m}^{\GO(a,\nu)}),
$$
and
$$
\GO(a,\ve_{m}^{\nu}\cdot\psi)=\\ 
\begin{cases}
\ve_{m}^{\nu}\cdot \GO(a,\psi)+\GO(a,\ve_{m}^{\nu}) & \text{ if }\hspace{0.5 cm}  \psn{\nu }\leq a,\\
\ve_{m}^{\GO(a,\nu)}\cdot \GO(a,\ve_{m})+\GO(a,\ve_{m}^{\GO(a,\nu)}) & \text{ if }\hspace{0.5 cm}  \psn{\nu}>a,
\end{cases}
$$
where $\ve_{m}^{\nu}$ and $\ve_{m}^{\nu}\cdot\psi$ are written in normal form and $\psi>0$ is a limit. 
Finally, if $\alpha=\xi+\ve_{m}^{\nu}\cdot \psi$ in short normal form, with $\xi\neq 0$, we consider two cases:
\begin{enumerate}
\item If $\psi=\phi+1$,  then write  $\alpha=\xi'+\ve_{m}^{\nu}$, where $\xi'=\xi+\ve_{m}^{\nu}\cdot\phi$, and define
$$
\GO(a, \alpha)=
\begin{cases}
\xi' + \GO(a,\ve_{m}^{\nu}) & \text{ if }\hspace{0.5 cm}  \psn{\xi' }\leq a,\\
\GO(a,\xi') & \text{ if }\hspace{0.5 cm}  \psn{\xi'}>a.
\end{cases}
$$
\item If $\psi$ is a limit ordinal,  then
$$
\GO(a, \alpha)=
\begin{cases}
\xi + \GO(a,\ve_{m}^{\nu}\cdot \psi) & \text{ if }\hspace{0.5 cm}  \psn{\xi }\leq a,\\
\GO(a,\xi) & \text{ if }\hspace{0.5 cm}  \psn{\xi}>a.
\end{cases}
$$
\end{enumerate}  We extend Lemma 3.6 of \cite{Bigorajska2006}.

\begin{lem}
For every $a>0$ and every $\alpha>0$ we have: for all $\gamma$ if $\gamma < \alpha$ and $\psn{\gamma}\leq a$, then $\gamma \leq \GO(a, \alpha)$.
\end{lem}
\begin{proof}
We extend the proof given in \cite{Bigorajska2006} for $\alpha<\ve_{0}$ and allow $\alpha\geq \ve_{0}$. \par
So, let $\alpha=\ve_{m}$, with $m\geq 0$, and $\gamma < \alpha$. If in addition $\psn{\gamma}\leq a$, then $\gamma < \tow_{a}(\ve_{m-1})$, by the definition of psn. We apply the induction hypothesis to obtain
$$\gamma \leq \GO(a,\tow_{a}(\ve_{m-1}))=\GO(a,\ve_{m}).
$$
\par
Now, let $\alpha=\ve_{m}^{\nu}$ with $\nu>1$ and assume that for each $\beta<\alpha$ the lemma holds. Let $\gamma<\ve_{m}^{\nu}$ and $\psn{\gamma}\leq a$. Write $\gamma = \ve_{m}^{\mu}\cdot\xi+\psi$ with $\ve_{m}^{\mu}\ggg \psi$ and $\xi<\ve_{m}$. Then $\mu<\nu$, because $\gamma <\alpha$. Since $m\geq l(\gamma)-1$, Lemma \ref{lem-bases} implies $\psn{\xi}\leq \psn{\gamma} \leq a$. Now apply the inductive assumption to $\nu$ to obtain $\mu \leq \GO(a,\nu)$. Because $\psn{\xi}\leq a$ and $\xi<\ve_{m}$, the induction hypothesis applied to $\ve_{m}$ yields $\xi\leq \GO(a,\ve_{m})$. Moreover $\psn{\psi}\leq \psn{\gamma} \leq a$ and $\psi<\ve_{m}^{\mu}$, so $\psi<\ve_{m}^{\GO(a,\nu)}$. Thus by the induction hypothesis applied to this ordinal, we get $\psi\leq\GO(a,\ve_{m}^{\GO(a,\nu)}) $. All together we obtain
$$\gamma = \ve_{m}^{\mu}\cdot\xi+\psi\leq \ve_{m}^{\GO(a,\nu)}\cdot \GO(a,\ve_{m})+\GO(a,\ve_{m}^{\GO(a,\nu)}).$$
\par
Consider $\alpha=\ve_{m}^{\nu}\cdot \psi$, where $\alpha$ is written in normal form and $\psi$ is a limit. Assume that for each $\beta<\alpha$ the lemma holds. Let $\gamma<\ve_{m}^{\nu}\cdot \psi$ and $\psn{\gamma}\leq a$. Write $\gamma = \ve_{m}^{\mu}\cdot\xi+\phi$ with $\ve_{m}^{\mu}\ggg \phi$ and $\xi<\ve_{m}$. We consider two different cases:

\begin{enumerate}
\item $\mu=\nu :$ Then $\xi<\psi$ and, since $\psn{\xi}\leq \psn{\gamma}\leq a$, we may apply the induction hypothesis to the ordinal $\psi$ to obtain $\xi \leq \GO(a,\psi)$.  Furthermore $\phi<\ve_{m}^{\mu}=\ve_{m}^{\nu}$. Applying the induction hypothesis to this last ordinal yields $\phi \leq \GO(a,\ve_{m}^{\nu})$. Combining the previous, we find
$$\gamma= \ve_{m}^{\mu}\cdot\xi+\phi\leq 
\ve_{m}^{\nu}\cdot \GO(a,\psi)+\GO(a,\ve_{m}^{\nu})=\GO(a,\ve_{m}^{\nu}\cdot\psi),$$
since $\psn{\nu}=\psn{\mu}\leq \psn{\gamma}\leq a$.
\item $\mu<\nu :$ Since $\psn{\mu}\leq \psn{\gamma}\leq a$, we can apply the induction hypothesis to $\nu$ and obtain $\mu\leq \GO(a,\nu)$. Since $\xi<\ve_{m}$ and $\psn{\xi}\leq a$, we have by the induction hypothesis $\xi \leq  \GO(a,\ve_{m})$. Furthermore, since $\phi < \ve_{m}^{\mu}\leq \ve_{m}^{\GO(a,\nu)}$ and $\psn{\phi}\leq \psn{\gamma}\leq a$, the induction hypothesis yields $\phi\leq \GO(a,\ve_{m}^{\GO(a,\nu)})$. We obtain
$$\gamma= \ve_{m}^{\mu}\cdot\xi+\phi\leq 
\ve_{m}^{\GO(a,\nu)}\cdot \GO(a,\ve_{m})+\GO(a,\ve_{m}^{\GO(a,\nu)})\leq\GO(a,\ve_{m}^{\nu}\cdot\psi),$$
\end{enumerate}
which completes this case.
\par
Finally, let $\alpha=\xi+\ve_{m}^{\nu}\cdot \psi$. We consider two different cases. If $\psi$ is a limit ordinal the proof is analogous to the one  for $\alpha=\xi+\omega^{\nu}$. If $\psi=\phi+1$, then first write $\alpha=\xi'+\ve_{m}^{\nu}$, where $\xi'=\xi+\ve_{m}^{\nu}\cdot \phi$, and now proceed in the same way as in the first case.
\end{proof}

\begin{defin}
Define $F:(<\vei) \to (<\vei)$ by the following conditions:
\begin{enumerate}
\item $F(0)=0;$
\item $F(\alpha+1)=F(\alpha)+1;$
\item $\beta \gg \alpha \Rightarrow F(\beta+\alpha)=F(\beta)\oplus F(\alpha);$
\item $F(\omega^{n})=\omega^{n}+\omega^{n-1}+\ldots+\omega^{0}$ for $n<\omega$;
\item $F(\omega^{\alpha})=\omega^{\alpha}\cdot 2+1$ for $\alpha\geq \omega$.
\end{enumerate}

\end{defin}
Let us write an explicit formula for $F$. We write
$$\alpha=\omega^{\alpha_{0}}\cdot a_{0}+\ldots +\omega^{\alpha_{s}}\cdot a_{s}+\omega^{n}\cdot m_{n}+\ldots+\omega^{0}\cdot m_{0},$$
in normal form to the base $\omega$, with the difference that we allow some $m_{i}$'s to be zero. Then $F(\alpha)$ is equal to
\begin{align*} \omega^{\alpha_{0}}\cdot 2a_{0}+\ldots +\omega^{\alpha_{s}}\cdot 2a_{s}&\\
+\omega^{n}\cdot m_{n}+\omega^{n-1}\cdot (m_{n}+m_{n-1})+\ldots+\omega^{0}\cdot (m_{n}+\ldots+m_{0})&\\
+(a_{0}+\ldots+a_{s})&.
\end{align*}

In the next lemma we investigate the relation between the pseudonorm of $\alpha$ and the one of $F(\alpha)$. 
\begin{lem}
Let $\alpha<\vei$ with $l(\alpha)-1=m$ and $a=\psn{\alpha}$. Then 
$$\psn{F(\alpha)}\leq2\underbrace{\tow_{a}(\ldots(\tow_{a}}_{m+2\text{ times}} (a+1))\ldots).$$
\end{lem}
\begin{proof}
By induction on the complexity of $\alpha$.
\par
If $\alpha=0$, then $\psn{\alpha}=0$ and so
$$\psn{F(\alpha)}=\psn{0}=0<\tow_{a} (a+1).$$
\par
If $\alpha=\omega^{n}$, with $n \leq \omega$, and write $\alpha$ in its normal form $\alpha=\ve_{-1}^{n}\cdot1$. We have
$$F(\alpha)=F(\omega^{n})=\omega^{n}+\omega^{n-1}+\ldots+\omega^{0}.$$
So the normal form of $F(\alpha)$ equals $\ve_{-1}^{n}\cdot1+\ve_{-1}^{n-1}\cdot1+\ldots+\ve_{-1}^{0}\cdot1$. This yields
\begin{align*}
\psn{F(\alpha)} & \leq  \max \{h(F(\alpha)), \psn{n}, \psn{n-1},\ldots, \psn{0}, \psn{1}\}\\
& = \max \{h(\alpha), n,1\}\\
&=  \psn{\alpha}\\
&=a\\
&< 2\tow_{a} (a+1).
\end{align*}
\par
Let $\alpha=\omega^{\beta}$, with $\beta\geq \omega$, and write $\alpha$ in its normal form $\alpha=\ve_{m}^{\alpha_{0}}\cdot\xi_{0}$, for some $m\geq -1$. We have
$$F(\alpha)=F(\omega^{\beta})=\omega^{\beta}\cdot 2+1=\alpha\cdot 2+1.$$
So the normal form of $F(\alpha)$ equals $\ve_{m}^{\alpha_{0}}\cdot\xi_{0}\cdot 2+\ve_{m}^{0}$. This yields
\begin{align*}
\psn{F(\alpha)} & \leq  \max \{h(F(\alpha)), \psn{\alpha_{0}}, \psn{\xi_{0}\cdot 2}, \psn{0}, \psn{1}\}\\
& \leq \max \{h(\alpha), \psn{\alpha_{0}}, 2\cdot \psn{\xi_{0}}\}\\
&\leq  2 \cdot \psn{\alpha}\\
& = 2a\\
&< \underbrace{\tow_{a}(\ldots(\tow_{a}}_{m+2\text{ times}} (a+1))\ldots).
\end{align*}
\par
Let $\alpha=\beta+\gamma$, with $\ve_{0}>\beta \gg \gamma$,  $\beta=_{NF}\omega^{\beta_{0}}\cdot n_{0}+\ldots+\omega^{\beta_{s}}\cdot n_{s}$ and $\gamma=_{NF}\omega^{\gamma_{0}}\cdot m_{0}+\ldots+\omega^{\gamma_{u}}\cdot m_{u}$, both in normal form. First consider the case in which $\gamma_{u}\geq \omega$. Write $\beta$ in the following way
\begin{align*}
\beta& = \omega^{\beta_{0}}\cdot n_{0}+\ldots+\omega^{\beta_{s}}\cdot n_{s}\\
&=\sum_{i=0}^{s}\sum_{j=0}^{n_{i}}\omega^{\beta_{ij}},
\end{align*}
where $\beta_{ij}=\beta_{i}$, for every $i \in \{0,\ldots s\}$.
We do the same for $\gamma$,
\begin{align*}
\gamma& = \omega^{\gamma_{0}}\cdot m_{0}+\ldots+\omega^{\gamma_{u}}\cdot m_{u}\\
&=\sum_{i=0}^{u}\sum_{j=0}^{m_{i}}\omega^{\gamma_{ij}},
\end{align*}
where $\gamma_{ij}=\gamma_{i}$, for every $i \in \{0,\ldots u\}$. We use these notations to the base $\omega$ to obtain an explicit form of $F(\alpha)$. Let $n=\max \{ n_{i}: 0\leq i \leq s\}$ and $m = \max \{ m_{i} : 0\leq i \leq u\}$. Then
\begin{align*}
\psn{F(\alpha)}&=\psn{F(\beta+\gamma)}\\
&=\psn{F(\beta)\oplus F(\gamma)}\\
&=\psn{F(\sum_{i=0}^{s}\sum_{j=0}^{n_{i}}\omega^{\beta_{ij}})\oplus F(\sum_{i=0}^{u}\sum_{j=0}^{m_{i}}\omega^{\gamma_{ij}})}\\
&=\psn{\bigoplus_{i=0}^{s}\bigoplus_{j=0}^{n_{i}}F(\omega^{\beta_{ij}})\oplus\bigoplus_{i=0}^{u}\bigoplus_{j=0}^{m_{i}}F(\omega^{\gamma_{ij}})}\\
&=\psn{\bigoplus_{i=0}^{s}\bigoplus_{j=0}^{n_{i}}(\omega^{\beta_{ij}}\cdot 2+1)\oplus\bigoplus_{i=0}^{u}\bigoplus_{j=0}^{m_{i}}(\omega^{\gamma_{ij}}\cdot 2+1)}\\
&=\psn{ \sum_{i=0}^{s}\sum_{j=0}^{n_{i}}\omega^{\beta_{ij}}\cdot 2 +  \sum_{i=0}^{u}\sum_{j=0}^{m_{i}}\omega^{\gamma_{ij}}\cdot 2 + \sum_{i=0}^{s}n_{i} + \sum_{i=0}^{u}m_{i}}\\
& = \psn{\beta\cdot 2 + \gamma\cdot 2 + \sum_{i=0}^{s}n_{i} + \sum_{i=0}^{u}m_{i}}\\
& = \psn{\alpha\cdot 2  + \sum_{i=0}^{s}n_{i} + \sum_{i=0}^{u}m_{i}}\\
& \leq2 \cdot \psn{\alpha}  + \sum_{i=0}^{s}n_{i} + \sum_{i=0}^{u}m_{i}\\
& \leq2 \cdot \psn{\alpha}  + sn+um\\
& \leq2 \cdot \psn{\alpha}  + s\cdot \psn{\beta}+u\cdot \psn{\gamma}\\
& \leq2 \cdot \psn{\alpha}  + (s+u)\cdot \psn{\alpha}\\
& \leq2a+ ( \tow_{a-1}(a+1)+ \tow_{a-1}(a+1))\cdot a\\
& =2a+2a \cdot \tow_{a-1}(a+1)\\
& \leq2\tow_{a}(a+1).
\end{align*}
The first, the third and the fourth inequality are justified by the definition of the pseudonorm. The second by the definiton of $n$ and $m$. The second last holds because of
\begin{align*}
s&\leq|\{ \zeta < \tow_{a-1}(\omega) : \psn{\zeta}\leq \psn{\beta}\}|\\
&\leq|\{ \zeta < \tow_{a-1}(\omega) : \psn{\zeta}\leq a\}|\\
&= \tow_{a-1}(a+1),
\end{align*} 
where the equality is due to Lemma \ref{lem-Nmk}, and in the same way $u \leq \tow_{a-1}(a+1)$. To see this one uses a similar counting argument as done in Lemma \ref{lem-Nmk}.
\par
Let $\alpha=\beta+\gamma$, with $\ve_{0}>\beta \gg \gamma$,  $\beta=_{NF}\omega^{\beta_{0}}\cdot n_{0}+\ldots+\omega^{\beta_{s}}\cdot n_{s}$ and $\gamma=_{NF}\omega^{\gamma_{0}}\cdot m_{0}+\ldots+\omega^{\gamma_{u}}\cdot m_{u}$, both in normal form. Now allow some or all $\beta_{i}$ and $\gamma_{i}$ to be smaller than $\omega$. This case is dealt with in the same way as the previous one, but one should pay attention to possibly different forms of $F(\beta)$ and $F(\gamma)$, since now $\omega^{n}\cdot m$ may occur, where $n,m \in \Nat$.
\par
Let $\alpha=\beta+\gamma$, with $\beta \gg \gamma\geq\ve_{0}$, $\beta=_{NF}\ve_{m}^{\beta_{0}}\cdot \xi_{0}+\ldots+\ve_{m}^{\beta_{s}}\cdot \xi_{s}$ in normal form and $\gamma=\ve_{m}^{\gamma_{0}}\cdot \eta_{0}+\ldots+\ve_{m}^{\gamma_{u}}\cdot \eta_{u}$. First we consider the case in which $\ve_{m}^{\gamma_{u}}\cdot \eta_{u}\geq \omega^{\omega}$. We write $\beta$  to the base $\omega$,
\begin{align*}
\beta& = \ve_{m}^{\beta_{0}}\cdot \xi_{0}+\ldots+\ve_{m}^{\beta_{s}}\cdot \xi_{s} \\
& = \omega^{\ve_{m}\cdot\beta_{0}}\cdot \xi_{0}+\ldots+\omega^{\ve_{m}\cdot\beta_{s}}\cdot \xi_{s} \\
&=\sum_{i=0}^{s}\sum_{j=0}^{r_{i}}\omega^{\ve_{m}\cdot\beta_{i}+\xi_{i}^{j}},
\end{align*}
where $\xi_{i}=\sum_{j=0}^{r_{i}}\omega^{\xi_{i}^{j}}$, for every $i \in \{0,\ldots s\}$. In a similar way we obtain 
$$
\gamma = \sum_{i=0}^{u}\sum_{j=0}^{t_{i}}\omega^{\ve_{m}\cdot\gamma_{i}+\eta_{i}^{j}},
$$
where $\eta_{i}=\sum_{j=0}^{t_{i}}\omega^{\eta_{i}^{j}}$, for every $i \in \{0,\ldots u\}$. We use these notations to the base $\omega$ to obtain an explicit form of $F(\alpha)$. Let $r=\max \{ r_{i}: 0\leq i \leq s\}$ and $t = \max \{ t_{i} : 0\leq i \leq u\}$. Then
\begin{align*}
\psn{F(\alpha)}&=\psn{F(\beta+\gamma)}\\
&=\psn{F(\beta)\oplus F(\gamma)}\\
&=\psn{F(\sum_{i=0}^{s}\sum_{j=0}^{r_{i}}\omega^{\ve_{m}\cdot\beta_{i}+\xi_{i}^{j}})\oplus F(\sum_{i=0}^{u}\sum_{j=0}^{t_{i}}\omega^{\ve_{m}\cdot\gamma_{i}+\eta_{i}^{j}})}\\
&=\psn{\bigoplus_{i=0}^{s}\bigoplus_{j=0}^{r_{i}}F(\omega^{\ve_{m}\cdot\beta_{i}+\xi_{i}^{j}})\oplus\bigoplus_{i=0}^{u}\bigoplus_{j=0}^{t_{i}}F(\omega^{\ve_{m}\cdot\gamma_{i}+\eta_{i}^{j}})}\\
&=\psn{\bigoplus_{i=0}^{s}\bigoplus_{j=0}^{r_{i}}(\omega^{\ve_{m}\cdot\beta_{i}+\xi_{i}^{j}}\cdot 2+1)\oplus\bigoplus_{i=0}^{u}\bigoplus_{j=0}^{t_{i}}(\omega^{\ve_{m}\cdot\gamma_{i}+\eta_{i}^{j}}\cdot 2+1)}\\
&= \psn{\sum_{i=0}^{s}\sum_{j=0}^{r_{i}}\omega^{\ve_{m}\cdot\beta_{i}+\xi_{i}^{j}}\cdot 2 +  \sum_{i=0}^{u}\sum_{j=0}^{t_{i}}\omega^{\ve_{m}\cdot\gamma_{i}+\eta_{i}^{j}}\cdot 2 + \sum_{i=0}^{s}r_{i} +  \sum_{i=0}^{u}t_{i}}\\
& = \psn{ \beta\cdot 2 + \gamma\cdot 2 +  \sum_{i=0}^{s}r_{i} +  \sum_{i=0}^{u}t_{i}}\\
& = \psn{ \alpha \cdot 2 +  \sum_{i=0}^{s}r_{i} +  \sum_{i=0}^{u}t_{i}}\\
& \leq 2\cdot \psn{ \alpha} +  \sum_{i=0}^{s}r_{i} +  \sum_{i=0}^{u}t_{i}\\
& \leq 2\cdot \psn{ \alpha} +  rs+tu\\
& \leq 2a+2 \cdot \tow_{a}\underbrace{(\tow_{a}(\ldots(\tow_{a}}_{m\text{ times}} (a+1))\ldots)) \cdot  \tow_{a-1}\underbrace{(\tow_{a}(\ldots(\tow_{a}}_{m+1\text{ times}} (a+1))\ldots)) \\
& \leq 2\underbrace{\tow_{a}(\ldots(\tow_{a}}_{m+2\text{ times}} (a+1))\ldots).
\end{align*}
The first inequality holds by the definition of the pseudonorm. The second by the definition of $r$ and $t$.
The third inequality holds because of
\begin{align*}
s&\leq|\{ \zeta < \tow_{h(\beta)-1}(\ve_{m}) : \psn{\zeta}\leq \psn{\beta}\}|\\
&\leq|\{ \zeta < \tow_{a-1}(\ve_{m}): \psn{\zeta}\leq a\}|\\
&= \tow_{a-1}\underbrace{(\tow_{a}(\ldots(\tow_{a}}_{m+1\text{ times}} (a+1))\ldots)),
\end{align*} 
where the equality is due to Lemma \ref{lem-Nmk}, and in the same way 
$$u \leq \tow_{a-1}\underbrace{(\tow_{a}(\ldots(\tow_{a}}_{m+1\text{ times}} (a+1))\ldots)).$$ To see this one uses a similar counting argument as done in Lemma \ref{lem-Nmk}. We also have used
\begin{align*}
r&\leq|\{ \zeta < \tow_{\psn{\beta}}(\ve_{m-1}) : \psn{\zeta}\leq \psn{\beta}\}|\\
&\leq|\{ \zeta < \tow_{a}(\ve_{m-1}): \psn{\zeta}\leq a\}|\\
&= \tow_{a}\underbrace{(\tow_{a}(\ldots(\tow_{a}}_{m\text{ times}} (a+1))\ldots)),
\end{align*} 
and similarly
$$u \leq  \tow_{a}\underbrace{(\tow_{a}(\ldots(\tow_{a}}_{m\text{ times}} (a+1))\ldots)).$$
\par
Now consider the general case in which $\alpha=\beta+\gamma$, with $\beta \gg \gamma$ and no further restrictions whatsoever.
This case is dealt with in the same way as the previous ones, but one should pay attention to possibly different forms of $F(\beta)$ and $F(\gamma)$.
\end{proof}

\begin{cor} \label{cor-psnF}
If $l(\alpha) \leq m+1$, then 
$$\psn{F(\GO(a,\alpha))}\leq 2\underbrace{\tow_{a}(\ldots(\tow_{a}}_{m+2\text{ times}} (a+1))\ldots).$$
\end{cor}
\begin{proof}
Remark that $\psn{\GO(a,\alpha)}=a$ and $l(\GO(a,\alpha))\leq m+1$. Then use the previous lemma. 
\end{proof}

\begin{lem}
For all $\alpha$, for all $\beta \gg \alpha$, and for all $A$ and $G$ if $G: A \to (\leq \beta + \alpha)$ is strictly decreasing $\min(A)>1$ and $\forall x \in A$, $\psn{G(x)}\leq x$, then setting $w=\max \{ x \in A : G(x)\geq \beta \}$ we have: $\{ x \in A : x \leq w \}$ is at most $F(\alpha)$-large.
\end{lem}

\begin{proof}
The proof is an extension of the one in \cite{Bigorajska2006}. As we allow $\alpha$ and $\beta$ to be larger than $\ve_{0}$, we need to consider extra cases. Let $D=\{x \in A : G(x)\geq \beta \}$.
\par
\textsc{Case 1':} If $\lambda = \ve_{0}$, then we have $G(a_{0})\leq \beta + \ve_{0}$, and so 
$$G(a_{1})\leq \beta + \GO(a_{1},\ve_{0})=\beta+\GO(a_{1},\tow_{a_{1}}(\omega)),$$
which is exactly \textsc{Case 4}, already dealt with in \cite{Bigorajska2006}.
\par
\textsc{Case 1'' :} $\lambda = \ve_{m}$ with $m>0$. As usual we have $G(a_{0})\leq \beta + \ve_{m}$, and so 
$$G(a_{1})\leq \beta + \GO(a_{1},\ve_{m})=\beta+\GO(a_{1},\tow_{a_{1}}(\ve_{m-1})).$$
We apply the induction hypothesis to $\alpha=\GO(a_{1},\tow_{a_{1}}(\ve_{m-1}))$ and infer that $D\setminus \{a_{0}\}$ is at most $F(\alpha)$-large. Assume that $D$ is not at most $F(\tow_{a_{1}}(\ve_{m-1}))$-large. Remark that $\tow_{a_{1}}(\ve_{m-1})=\omega^{\ve_{m-1}\cdot\tow_{a_{1}-1}(\ve_{m-1})}$, so $D$ is not at most $\tow_{a_{1}}(\ve_{m-1})\cdot 2 + 1$-large. Then $D\setminus \{a_{0}\}$ is not at most $\tow_{a_{1}}(\ve_{m-1})\cdot 2 $-large, so $E=D\setminus \{a_{0}, \max D \}$ is $\tow_{a_{1}}(\ve_{m-1})\cdot 2 $-large. Let $z=h^{E}_{\tow_{a_{1}}(\ve_{m-1})}(a_{1})$. Then 
\begin{align*}
z&= h^{E}_{\tow_{a_{1}}(\ve_{m-1})}(a_{1})\\
&> h^{E}_{\ve_{m-1}}(a_{1})\\&> 2\underbrace{\tow_{a_{1}}(\ldots(\tow_{a_{1}}}_{m+1\text{ times}} (a_{1}+1))\ldots)\\
&\geq \psn{F(\GO(a_{1},\tow_{a_{1}}(\ve_{m-1})))},
\end{align*}
using the results of Corollary \ref{cor-psnF} and Lemma \ref{lem-htow}. So $$\tow_{a_{1}}(\ve_{m-1})\Rightarrow_{z}F(\GO(a_{1},\tow_{a_{1}}(\ve_{m-1}))),$$ hence $h^{E}_{F(\GO(a_{1},\tow_{a_{1}}(\ve_{m-1})))}(z)\downarrow$. From this we infer that 
$$h^{E}_{F(\GO(a_{1},\tow_{a_{1}}(\ve_{m-1})))}(a_{1})\downarrow,$$ and we get a contradiction with the fact that $D\setminus \{ a_{0}\}$ is at most $F(\alpha)$-large.
\par
\textsc{Case 4' :} $\lambda = \ve_{m}^{\nu}$ with $l(\lambda)=m\geq0$ and $\nu>1$. Using the same notation as above we see that $G(a_{0})\leq \beta + \ve_{m}^{\nu}$, and so 
$$G(a_{1})\leq \beta + \GO(a_{1},\ve_{m}^{\nu}).$$
We apply the induction hypothesis to $\alpha= \GO(a_{1},\ve_{m}^{\nu})$
and infer that $D\setminus \{a_{0}\}$ is at most $F(\alpha)$-large. Assume that $D$ is not at most $F(\ve_{m}^{\nu})$-large. Remark that $\ve_{m}^{\nu}=\omega^{\ve_{m}\cdot \nu}$, so $D$ is not at most $\ve_{m}^{\nu}\cdot 2 + 1$-large. Then $D\setminus \{a_{0}\}$ is not at most $\ve_{m}^{\nu}\cdot 2 $-large, so $E=D\setminus \{a_{0}, \max D \}$ is $\ve_{m}^{\nu}\cdot 2 $-large. Let $z=h^{E}_{\ve_{m}^{\nu}}(a_{1})$. Remind that $l(\lambda)=m$. Then 
\begin{align*}
z&= h^{E}_{\ve_{m}^{\nu}}(a_{1})\\
&\geq h^{E}_{\ve_{m}}(a_{1})\\
&> 2\underbrace{\tow_{a_{1}}(\ldots(\tow_{a_{1}}}_{m+2\text{ times}} (a_{1}+1))\ldots)\\
&\geq \psn{F(\GO(a_{1},\ve_{m}^{\nu}))},
\end{align*}
using the results of Corollary \ref{cor-psnF} and Lemma \ref{lem-htow}. So $$\ve_{m}^{\nu}\Rightarrow_{z}F(\GO(a_{1},\ve_{m}^{\nu}))),$$ hence $h^{E}_{F(\alpha)}(z)\downarrow$. From this we infer that $h^{E}_{F(\alpha)}(a_{1})\downarrow$, and we get a contradiction with the fact that $D\setminus \{ a_{0}\}$ is at most $F(\alpha)$-large.
\end{proof}
Now we are finally ready to give the estimation lemma.

\begin{lem} \emph{(The estimation lemma)} For every $\alpha<\vei$ we have: for every $A \subseteq \Nat$ with $\min A>0$, if there exists a strictly decreasing function $G: A \to (\leq\alpha)$ such that $\psn{G(a)}\leq a$ for all $a \in A$, then $A$ is at most $F(\alpha)$-large. 
\end{lem}

\begin{proof}
Apply the previous lemma with $\beta=0$.
\end{proof}

\section{Partitioning $\alpha$-large sets}

\begin{defin} Let $\alpha$ and $\beta$ be ordinals below $\vei$. If $\alpha=\ve_{m}^{\alpha_{0}}\cdot\xi_{0}+\ldots+\ve_{m}^{\alpha_{s}}\cdot\xi_{s}$ 
(in normal form to the base $\ve_{m}$),
then define $v(\ve_{m};\alpha,\delta)$ as the coefficient of $\ve_{m}^{\delta}$ (and $v(\ve_{m};\alpha,\delta)=0$ if $\ve_{m}^{\delta}$ does not occur in the normal form). Now define
$$\LD(\ve_{m};\alpha,\beta):= \max \{ \delta < \vei : v(\ve_{m};\alpha,\delta)\neq v(\ve_{m};\beta,\delta)\}.$$
\end{defin}

\begin{defin}
Let $\Gamma \subseteq (<\vei)$ and $\Theta: \Gamma \to (<\alpha)$. We say that $\Theta$ is an ordinal (or $\alpha$-ordinal) estimating function if it is strictly increasing and for all $\gamma \in \Gamma$, $\psn{\Theta(\gamma)}\leq \psn{\gamma}$.
\end{defin}

Let us make the convention that when working with finite sets of ordinals we write them down in decreasing order.

\begin{defin} Let $\alpha,\beta, \gamma <\vei$.
$$
L_{3}(\ve_{m};\alpha, \beta,\gamma):=
\begin{cases}
0 & \text{ if   } \, \LD(\ve_{m};\alpha,\beta)<\LD(\ve_{m};\beta,\gamma), \\
1 & \text{ if   } \, \LD(\ve_{m};\alpha,\beta)= \LD(\ve_{m};\beta,\gamma),\\
2 & \text{ if   } \, \LD(\ve_{m};\alpha,\beta) > \LD(\ve_{m};\beta,\gamma). \\
\end{cases}
$$
\end{defin}

\begin{lem}\label{lem-L3} Let $\Gamma$ be a finite subset of $(<\vei)$ such that $\max \Gamma < \tow_{s}(\ve_{m})$.
\begin{enumerate}
\item If $L_{3}(\ve_{m})$ colors $[\Gamma]^{3}$ by 1, then there exists an ordinal estimating function $\Theta$ defined on $\Gamma$ with values in $(<\ve_{m}).$
\item If  $L_{3}(\ve_{m})$ colors $[\Gamma]^{3}$ by 2,  then there exists an ordinal estimating function $\Theta$ defined on $\Gamma\setminus \{\min \Gamma \}$ with values in $(<\tow_{s-1}(\ve_{m})).$
\item Let $\gamma_{0}=\max \Gamma$ and $\psn{\gamma_{0}}=a$. If $\psn{\gamma_{0}}\geq s-1$ and $L_{3}(\ve_{m})$ colors $[\Gamma]^{3}$ by 0, then 
$$|\Gamma|\leq \tow_{s-1}(\underbrace{\tow_{a}(\ldots(\tow_{a}}_{m+1 \text{ times}} (a+1))\ldots))+1.$$
\end{enumerate}
\end{lem}

\begin{proof} Let $\Gamma=\{\gamma_{0}, \ldots, \gamma_{r-1}\}$ be a decreasing enumeration of $\Gamma$.  Since $\max \Gamma < \tow_{s}(\ve_{m})$, we have $l(\gamma_{i})\leq m+1$ for every $i \in \{0,...,r-1\}$.
\begin{enumerate} 
\item  Let $\gamma_{i}=\rho+\ve_{m}^{\delta_{i}}\cdot \xi_{i}+\tau_{i}$, be written in normal form to the base $\ve_{m}$. In this case neither $\rho$ nor $\delta_{i}$ depends on $i$. Put $\Theta(\gamma_{i})=\xi_{i}$. $\Theta$ has the required properties (use Lemma \ref{lem-bases}).
\item  Put $\delta_{i}=\LD(\ve_{m};\gamma_{i},\gamma_{i+1})$ and $\Theta(\gamma_{i})=\delta_{i}$. $\Theta$ has the required properties (use Lemma \ref{lem-bases}). In particular, its values are smaller than $\tow_{s-1}(\ve_{m})$. 

\item Put $\delta_{i}=\LD(\ve_{m};\gamma_{i},\gamma_{i+1})$. We assert that $\ve_{m}^{\delta_{i}}$ occurs with a nonzero coefficient in the Cantor normal form expansion to the base $\ve_{m}$ of $\gamma_{0}$ for every $i<r-1$. Indeed, fix $i<r-1$. We write  $\gamma_{i}=\rho_{i}+\ve_{m}^{\delta_{i}}\cdot \xi_{i}+\tau_{i}$ and compare this with the expansion of $\gamma_{i+1}$ to the base $\ve_{m}$. We see that $\xi_{i}$ must be greater than the coefficient at $\ve_{m}^{\delta_{i}}$ in the expansion of $\gamma_{i+1}$. In particular, $\xi_{i}>0$. Then  
\begin{align*}
|\Gamma|& = |\{ \delta_{i} : i < r-1\}|+1\\
&\leq |\{ \alpha < \tow_{s-1}(\ve_{m}):\psn{\alpha}\leq a\}|+1\\
&= \tow_{s-1}\underbrace{(\tow_{a}(\ldots(\tow_{a}}_{m+1\text{ times}} (a+1))\ldots))+1,
\end{align*}
where the last equality holds by Lemma \ref{lem-Nmk}. (Remark that we required $a\geq s-1$.)
\end{enumerate}
\end{proof}

\begin{lem}\label{lem-existence}
Let $m\geq0$. Then for every $k\geq 3$ there existst a partition $L_{k}(\ve_{m})$ of $[\vei]^{k}$ into $3^{k-2}$ parts such that for every $\Gamma \subseteq (<\vei)$ homogeneous for $L_{k}(\ve_{m})$, if $\max \Gamma < \tow_{k-1}(\ve_{m})$ and $\psn{\max \Gamma}\geq k-2$, then letting $\Gamma'$ be $\Gamma$ without the last $\frac{(k-2)(k-1)}{2}$ elements we have: there exists an ordinal estimating function $\Theta: \Gamma' \to (<\ve_{m})$ or $|\Gamma| \leq \tow_{k-2}(\underbrace{\tow_{a}(\ldots(\tow_{a}}_{m+1 \text{ times}} (a+1))\ldots))+k-2,$
where $a=\psn{\max \Gamma}$.
\end{lem}

\begin{proof}
By induction on $k$. 
\par
$k=3:$ See Lemma \ref{lem-L3} above. 
\par
$k \to k+1:$ in the same way as done in \cite{Bigorajska2006}. Assume the result for $k$. We construct the partition for $k+1$. Let $\alpha = (\alpha^{0}, \ldots, \alpha^{k})$ be a $(k+1)$-tuple of ordinals below $\ve_{\omega}$. We begin by putting $G(\alpha)=L_{3}(\ve_{m}; \alpha^{0},\alpha^{1},\alpha^{2})$, so this does not depend on the last $k-2$ coordinates of the sequence $\alpha$. We have: if $\Gamma$ is homogeneous for $G$, then $\Gamma''$ is homogeneous for $L_{3}(\ve_{m})$, where $\Gamma''$ denotes $\Gamma$ without its last $k-2$ elements.
\par
For $\alpha$ as above we let $\delta^{i}=\LD(\ve_{m};\alpha^{i},\alpha^{i+1})$. This $k$-tuple is not necessarily monotonic. Let
$$
W(\alpha)=
\begin{cases}
L_{k}(\ve_{m};\delta^{0},\ldots,\delta^{k-1})& \text{ if } \delta^{0}>\ldots>\delta^{k-1},\\
L_{k}(\ve_{m};\delta^{k-1},\ldots,\delta^{0}) &\text{ if } \delta^{0}<\ldots<\delta^{k-1},\\
0& \text{ in other cases}.
\end{cases}
$$
Finally we put $L_{k+1}(\ve_{m};\alpha)=\left<G(\alpha),W(\alpha)\right>$. We assert that this partition has the right properties. So let $\Gamma=\{\gamma_{0},\ldots,\gamma_{r-1}\}$ written in decreasing order and monochromatic for $L_{k+1}(\ve_{m})$ and let $\Gamma''$ denote $\Gamma$ without its last (\text{i.e. smallest}) $k-2$ elements. Also let $\Gamma'$ denote the set under consideration, \text{i.e.} $\Gamma$ without its last  $\frac{k(k-1)}{2}$ elements.
\par
\textsc{Case 1}. $G$ colours $[\Gamma]^{k+1}$ by 1. Then $L_{3}(\ve_{m})$ colours $[\Gamma'']^{3}$ by 1, so by part \text{1.} of Lemma \ref{lem-L3} there exists a function $\ve_{m}$-ordinal estimating $\Gamma''$. The domain of this function contains the required set $\Gamma'$. 
\par
\textsc{Case 2}. $G$ colours $[\Gamma]^{k+1}$ by 2. Then $L_{3}(\ve_{m})$ colours $[\Gamma'']^{3}$ by 2. Let $\delta_{i}=\LD(\ve_{m};\gamma_{i},\gamma_{i+1})$  for $i< (r-1)-(k-2)$. By the assumption of the case, the set  $\Delta =\{\delta_{0},\ldots,\delta_{r-k}\}$ is written in decreasing order. By homogeneity of $\Gamma$ with respect to $W$, the set $\Delta$ is monochromatic with respect to $L_{k}(\ve_{m})$. We distinguish two subcases.
\par
\emph{Subcase} (a). $|\Delta|\leq \tow_{k-2}(\underbrace{\tow_{a}(\ldots(\tow_{a}}_{m+1 \text{ times}} (a+1))\ldots))+k-2,$
where $a=\psn{\max \Delta}$. Obviously $\psn{\max \Delta} \leq \psn{\max \Gamma''} = \psn{\max \Gamma}$, so 
$$|\Gamma''|\leq  \tow_{k-2}(\underbrace{\tow_{a'}(\ldots(\tow_{a'}}_{m+1 \text{ times}} (a'+1))\ldots))+k-1,$$
where  $a'=\psn{\max \Gamma}$. Thus $|\Gamma|$ is smaller than or equal to
\begin{align*}
 \tow_{k-2}(\underbrace{\tow_{a'}(\ldots(\tow_{a'}}_{m+1 \text{ times}} &(a'+1))\ldots))+k-1+k-2\\
 &\leq \tow_{k-2}(\underbrace{\tow_{a'}(\ldots(\tow_{a'}}_{m+1 \text{ times}} (a'+1))\ldots))+k-1+a'\\
 &\leq \tow_{k-1}(\underbrace{\tow_{a'}(\ldots(\tow_{a'}}_{m+1 \text{ times}} (a'+1))\ldots))+k-1,
 \end{align*}
so the second case in the lemma holds.
\par
\emph{Subcase} (b).  There exists a function $\Theta$ which ordinal $\ve_{m}$-estimates $\Delta$ without its last $\frac{(k-2)(k-1)}{2}$ elements. Consider the function $\gamma_{i}\mapsto\Theta(\delta_{i})$. This function $\ve_{m}$-ordinal estimates $\Gamma$ without its last $\frac{(k-2)(k-1)}{2}+1+k-2=\frac{k(k-1)}{2}$ elements. In order to see that this function indeed has the required property, remind that $l(\gamma_{i})-1\leq m$ and apply Lemma \ref{lem-bases}.
\par
\textsc{Case 3}. $G$ colours $[\Gamma]^{k+1}$ by 0 and so $L_{3}(\ve_{m})$ colours $[\Gamma'']^{3}$ by $0$. We know that $\max \Gamma''=\max \Gamma = \gamma_{0}< \tow_{k}(\ve_{m})$ and that $a=\psn{\gamma_{0}}\geq k-1$. Now part \text{3.} of Lemma \ref{lem-L3} yields
$$
|\Gamma''|\leq\tow_{k-1}(\underbrace{\tow_{a}(\ldots(\tow_{a}}_{m+1 \text{ times}} (a+1))\ldots))+1,$$
and so
$$|\Gamma| \leq \tow_{k-1}(\underbrace{\tow_{a}(\ldots(\tow_{a}}_{m+1 \text{ times}} (a+1))\ldots))+k-1.$$
This completes the induction step and consequently the proof.
\end{proof}

Let $\mu$ be an ordinal and $A$ be a subset of the natural numbers, such that $A \setminus \{\max A\}$ is $\mu$-small. We associate an ordinal $\KS^{A}(\mu, a)$ to $\mu$ and $a \in A$. For a proper definition, we refer to the remark after Theorem 2.6 in \cite{Bigorajska2006}. We need the following lemma.

\begin{lem} \label{lem-KS-psn}
Let $\mu$ be an ordinal, $A \subseteq \Nat$, such that $A \setminus \{\max A\}$ is $\mu$-small and $\psn{\mu}\leq \min A$. Then
$
\psn{\KS^{A}(\mu;a)}\leq a
$ for every $a \in A$.
\end{lem}
\begin{proof}
Let $S(\alpha)$ denote the following statement:
$$
(\forall \beta) [(\psn{\alpha} \leq a\, \& \, \alpha \Rightarrow_{a} \beta \, \& \, \beta \text{ is nonlimit}) \Rightarrow (\psn{\beta}\leq a)]
$$
We first prove $(\forall \alpha ) S(\alpha)$ by induction on $\alpha$. In all cases $\beta \leq \alpha$, since $\alpha \Rightarrow_{a} \beta$. If $\beta = \alpha$, then there is nothing to prove, so we will assume $\beta<\alpha$.
\par
The case $\alpha=0$ is trivial.
\par
Assume  $\alpha=\gamma+1$, $\psn{\alpha} \leq a$ and $\alpha \Rightarrow_{a} \beta$, with $\beta$ nonlimit. Then $\psn{\gamma}\leq\psn{\gamma+1}\leq a$ and 
$\alpha[a]=\gamma$, so $\gamma \Rightarrow_{a}\beta$. The induction hypothesis applied to $\gamma$ yields $\psn{\beta}\leq a$.
\par
Assume $\alpha$ is a limit ordinal and the statement holds for all $\gamma<\alpha$. We have two different cases:
\begin{enumerate}
\item $\psn{\alpha[a]}\leq a :$ Since $\alpha[a] \Rightarrow_{a}\beta$ and $\alpha[a]<\alpha$, we can apply the induction hypothesis and obtain $\psn{\beta}\leq a$.
\item $\psn{\alpha[a]} = a+1 :$ This can only happen in the case that $\alpha=\ve_{m}$ for some $m\leq 0$. We have
\begin{align*}
\alpha[a] & \text{ equals $\tow_{a}(\ve_{m})$};\\
\alpha[a][a] & \text{ equals $\alpha[a]$, with the top $\ve_{m}$ replaced by $\tow_{a}(\ve_{m-1})$};\\
\alpha[a][a][a] & \text{ equals $\alpha[a][a]$, with the top $\ve_{m-1}$ replaced by $\tow_{a}(\ve_{m-2})$};\\
\vdots & \\
\alpha\hspace{-0.2em}\underbrace{[a]\ldots[a]}_{m+2 \text{ times}} & \text{ equals $\alpha\hspace{-0.2em}\underbrace{[a]\ldots[a]}_{m+1 \text{ times}} $, with the top $\ve_{0}$ replaced by $\tow_{a}(\omega)$};\\
\alpha\hspace{-0.2em}\underbrace{[a]\ldots[a]}_{m+3 \text{ times}} & \text{ equals $\alpha\hspace{-0.2em}\underbrace{[a]\ldots[a]}_{m+2 \text{ times}} $, with the top $\omega$ replaced by $a$}.
\end{align*}
Put $\delta=\alpha\hspace{-0.2em}\underbrace{[a]\ldots[a]}_{m+3 \text{ times}}$. Since $\beta$ is nonlimit it does not appear in the previous list of ordinals, so $\delta \Rightarrow_{a} \beta$ (remind that $\alpha \Rightarrow_{a} \beta$). Moreover, one can easily verify that $\psn{\delta }\leq a$. The induction hypothesis applied to $\delta$ yields $\psn{\beta}\leq a$. 
\end{enumerate}
Now return to the actual statement we would like to prove. Let $A= \{a_{0}, \ldots, a_{r}\}$. Take any $a_{i}\in A$ and apply $i+1$ times the statement $S(\alpha)$ in the following way: 
\begin{itemize}
\item[1] $\alpha=\mu$, $a=a_{0}$ and $\beta=\KS^{A}(\mu;a_{0})$ yield $\psn{\KS^{A}(\mu;a_{0})}\leq a_{0}$;
\item[2] $\alpha=\KS^{A}(\mu;a_{0})$, $a=a_{1}$ and $\beta=\KS^{A}(\mu;a_{1})$ yield $\psn{\KS^{A}(\mu;a_{1})}\leq a_{1}$; \\
$\vdots$
\item[i+1] $\alpha=\KS^{A}(\mu;a_{i-1})$, $a=a_{i}$ and $\beta=\KS^{A}(\mu;a_{i})$ yield $\psn{\KS^{A}(\mu;a_{i})}\leq a_{i}$. 
\end{itemize}
We have obtained $\psn{\KS^{A}(\mu;a_{i})}\leq a_{i}$.
\end{proof}

\begin{thm} \label{thm-partition}
Let $k\geq 3$ and $m\geq0$. Let $A\subseteq \Nat$ be at most $\tow_{k-2}(\ve_{m})$-large with $\min A \geq k$ and $\min A>3$. Then there exists a partition $R_{k}: [A]^{k}\to 3^{k-2}$ such that every $D\subseteq A$ homogeneous for $R_{k}$ is at most $F(\ve_{m}) + \frac{(k-2)(k-1)}{2}$-large.

\end{thm}
\begin{proof}
Let $A$ satisfy the assumption. Let $L_{k}(\ve_{m})$ be a partition of $[\vei]^{k}$ with the properties described in Lemma \ref{lem-existence}. For $a=(a^{0}, \ldots, a^{k-1})$ we let
$$R_{k}(a)=L_{k}(\ve_{m};\KS^{A}(\tow_{k-1}(\ve_{m});a^{0}), \ldots, \KS^{A}(\tow_{k-1}(\ve_{m});a^{k-1}))$$
and verify that this partition has the desired properties. So let $D$ be a subset of $A$ which is homogeneous for $R_{k}$. Then $\Gamma=\{\KS^{A}(\tow_{k-1}(\ve_{m});d):d \in D\}$ is homogeneous for $L_{k}(\ve_{m})$. Let $\Gamma'$ denote $\Gamma$ without its $ \frac{(k-2)(k-1)}{2}$ smallest elements and let $D'$ be $D$ without its $ \frac{(k-2)(k-1)}{2}$ last (\text{i.e.} greatest) elements.
\par
\textsc{Case 1}. There exists an ordinal estimating function $\Theta: \Gamma' \to (<\ve_{m})$. Define $G:D'\to (\leq \ve_{m})$ by $G(d)=\Theta(\KS^{A}(\tow_{k-1}(\ve_{m});d))$, for every $d \in D'$. Then $G$ is strictly decreasing and for every $d \in D'$,
\begin{align*}
\psn{G(d)}&=\psn{\Theta(\KS^{A}(\tow_{k-1}(\ve_{m});d))}\\
&\leq \psn{\KS^{A}(\tow_{k-1}(\ve_{m});d)}\\
& \leq d.
\end{align*}
The last inequality holds because of Lemma \ref{lem-KS-psn}, since $A\setminus \{\max A\}$ is $\tow_{k-1}(\ve_{m})$-small and  
$$\psn{\tow_{k-1}(\ve_{m})}=k\leq \min A.$$
Then $D'$ is at most $F(\ve_{m})$-large by the estimation lemma, and the result follows. 
(A minor point should be remarked here. The greatest elements of $\Gamma$ correspond to the smallest elements of $D$, thus to obtain the required conclusion we apply Lemma 5 in \cite{Bigorajska1999}.)
\par
\textsc{Case 2}. Since $ \psn{\max \Gamma}=d_{0}=\min D  \geq \min A \geq k$, we have
\begin{align*}
|\Gamma| & \leq\tow_{k-2}(\underbrace{\tow_{d_{0}}(\ldots(\tow_{d_{0}}}_{m+1 \text{ times}} (d_{0}+1))\ldots))+k-2\\
&\leq\underbrace{\tow_{d_{0}}(\ldots(\tow_{d_{0}}}_{m+2 \text{ times}} (d_{0}+1))\ldots)\\
& \leq h_{\omega^{2}\cdot2d_{0}(m+2)}(d_{0})\\
& < h_{\ve_{m}}(d_{0}),
\end{align*}
where the last inequality is a result from Lemma \ref{lem-hepsm}. Since $|D|=|\Gamma|$, $D$ is $\ve_{m}$-small, and so $F(\ve_{m})$-small, for $m\geq 0$.
\end{proof}

\begin{cor} \label{cor-main}
Let $m,k \in \Nat$ and $ \alpha=F(\ve_{m}) + \frac{(k-2)(k-1)}{2}+1$. Let $A$ be such that $A \to (\alpha)^{k}_{3^{k-2}}$, $3\leq k\leq \min(A)$ and $3<\min A$. Then $A$ is $\tow_{k-2}(\ve_{m})$-large.
\end{cor}

\begin{proof} The following claim is proven in \cite{Bigorajska2006} by induction on $\alpha$.
\par
\textsc{Claim}: for every $B$, if $B$ is $\alpha$-small, then there exists $C$ such that $\max B < \min C$ and $B\cup C$ is exactly $\alpha$-large.  \par
Granted the claim we argue as follows. Assume that $A$ is $\tow_{k-2}(\ve_{m})$-small. Let $C$ be as in the claim, so $\min A < \max C$ and $A\cup C$ is exactly $\tow_{k-2}(\ve_{m})$-large. By Theorem \ref{thm-partition} there exists a partition $L$ of $[A \cup C]^{k}$ into $3^{k-2}$ parts without an  $F(\ve_{m}) + \frac{(k-2)(k-1)}{2}+1$-large monochromatic set. We restrict $L$ to $[A]^{k}$ and see that this restriction does not admit an $F(\ve_{m}) + \frac{(k-2)(k-1)}{2}+1$-large monochromatic set.
\end{proof}

This last corollary will be useful to prove a miniaturisation of the infinite Ramsey Theorem.

\bibliographystyle{alpha}
\bibliography{Database-General}

\end{document}